\documentclass{amsart}
\usepackage{amsmath, amsthm, amssymb}
\usepackage{verbatim}
\usepackage{tikz}
\usetikzlibrary{matrix}

\usepackage[mathscr]{eucal} 
\usepackage{mathrsfs} 
\usepackage{cancel}
\usepackage{bbm}

\def\XXint#1#2#3{{\setbox0=\hbox{$#1{#2#3}{\int}$}
     \vcenter{\hbox{$#2#3$}}\kern-.5\wd0}}

\usepackage{enumitem, hyperref}
\makeatletter
\def\namedlabel#1#2{\begingroup
    #2%
    \def\@currentlabel{#2}%
    \phantomsection\label{#1}\endgroup}

\renewcommand{\d}{\partial}

\newcommand{\dbar}{\overline{\d}}
\newcommand{\ii}{\sqrt{-1}}
\newcommand{\wt}[1]{\widetilde{#1}}

\newcommand{\veps}{\varepsilon}

\newcommand{\al}{\alpha} 
 
\newcommand{\be}{\beta}

\newcommand{\de}{\delta}

\newcommand{\om}{\omega}

\newcommand{\tc}{{\tt c}}

\newcommand{\Ga}{\Gamma}
\newcommand{\Om}{\Omega}

\newcommand{\cH}{\mathcal{H}}

\newcommand{\bB}{\mathbb{B}}

\newcommand{\bP}{\mathbb{P}}
\newcommand{\bR}{\mathbb{R}}

\newcommand{\bC}{\mathbb{C}}

\newcommand{\cali}[1]{\mathscr{#1}}

\newcommand{\Kc}{\cali{K}}

\newtheorem{thm}{Theorem}
\newtheorem{prop}[thm]{Proposition}
\newtheorem{lem}[thm]{Lemma}
\newtheorem{cor}[thm]{Corollary}

\theoremstyle{definition}
\newtheorem{defn}[thm]{Definition}
\newtheorem{remark}[thm]{Remark}

\numberwithin{thm}{section}
\numberwithin{equation}{section}

\renewcommand{\[}{\begin{equation}}
\renewcommand{\]}{\end{equation}}

\newcommand{\rv}{\;\rvert\;}

\newcommand{\wed}{\wedge}

\usepackage{scalerel}[2014/03/10]
\usepackage[usestackEOL]{stackengine}
\def\intavg{\,\ThisStyle{\ensurestackMath{%
    \stackinset{c}{0\LMpt}{c}{0\LMpt}{\SavedStyle-}{\SavedStyle\phantom{\int}}}%
    \setbox0=\hbox{$\SavedStyle\int\,$}\kern-\wd0}\int}

\title[Alexander-Taylor's inequality in complex Sobolev spaces]{Alexander-Taylor's inequality for capacities in complex Sobolev spaces}

\author{Ngoc Cuong Nguyen} 
\address{Department of Mathematical Sciences, KAIST, 291 Daehak-ro, Yuseong-gu, Daejeon 34141, South Korea}
\email{cuongnn@kaist.ac.kr}
\date{\today}

\author{Do Duc Thai}
\address{Department of Mathematics, Hanoi National University of Education, 136 XuanThuy str.,  Hanoi, Vietnam}
\email{doducthai@hnue.edu.vn}

\begin{document}

\maketitle

\begin{abstract}  We prove a sharp inequality between the Alexander-Taylor capacity and  the functional capacity in a complex Sobolev space on a compact K\"ahler manifold. The latter space and capacity were introduced   by Dinh, Sibony and Vigny.
\end{abstract}

\section{Introduction}
 
 In 1980s Alexander and Taylor \cite{AT84} proved a sharp quantitative comparison between  the relative capacity of  Bedford and Taylor \cite{BT82}  and the projective capacity of Alexander \cite{A79}. It  is now called the Alexander-Taylor inequality in the literature, which has become very useful. For example,  it gave in \cite{AT84} a simplified proof of an effective version  of Josefson's theorem, due to El Mir \cite{EM79},  about the equivalence between locally pluripolar sets and globally pluripolar sets in $\bC^n$. Subsequently, the  inequality provided a crucial link in proving almost sharp uniform estimates for complex Monge-Amp\`ere equations on both domains in $\bC^n$ and compact K\"ahler manifolds  via the capacity method \cite{Ko98, Ko05}. Its generalization to the equations with semi-positive background forms on compact K\"ahler manifolds  played equally important roles in \cite{GZ05, GZ-book} and \cite{EGZ09}.

Let $(X,\al)$ be a compact K\"ahler manifold of dimension $n$. Let $\om\geq 0$ be a semi-positive closed $(1,1)$-form such that $\int_X \om^n>0$. For simplicity it is normalized by $\int_X \om^n=1$. Denote by $PSH(X,\om)$ the family of $\om$-plurisubharmonic ($\om$-psh) functions on $X$. 
The global Alexander-Taylor capacity $T_\om(\cdot)$ introduced and studied by Guedj and Zeriahi \cite{GZ05}  (see also \cite{DS06b} and \cite{GZ-book}) is as follows. For a compact subset $K\subset X$,
\[ \notag\label{eq:AT-cap-intro}
	T_\om (K) = \exp \left( - \sup_X V_K^* \right),
\]
where 
$\label{eq:SZ-intro}
	V_K(z) = \sup\{v(z): v\in PSH(X,\om), v\leq 0 \text{ on }K\}
$
and $V_K^*$ is its upper semi-continuous regularization. The latter function is often  called the Siciak-Zaharjuta extremal function associated to $K$. 
  
We are interested in the subspace $W^*(X) \subset W^{1,2}(X,\bR)$ introduced by Dinh and Sibony \cite{DS} (see Section~2), which is  called the  complex Sobolev space.  For $f\in W^*(X)$, let $\Ga_f$  denote the non-empty set of all positive closed $(1,1)$-currents $T$ satisfying
\[\notag\label{eq:w-ineq-intro}
	df \wed d^c f  \leq T \quad\text{weakly in } X.
\] 
Then, the $W^*$-norm of $f$ is given by
\[\notag\label{eq:w-norm-intro}
	\|f\|_*^2 = \|f\|_{L^2(X)}^2 + \inf\{\|T\| : T \in \Ga_f\},
\]
where  
$
	\|T\| = \int_X T \wed\al^{n-1}.
$ 
Soon afterward, Vigny \cite{Vigny} showed that $W^*(X)$ is a Banach space with the $W^*$-norm. One special feature of this new space is that it takes  complex structures of underlying manifolds into account. Therefore, it is a suitable framework to work with complex objects such as holomorphic/meromorphic maps and quasi plurisubharmonic functions. Since then, this space has found numerous applications in complex dynamics in higher dimensions and some related areas  \cite{BiD23},  \cite{DKW}, \cite{DNV25}, \cite{Vig15}, \cite{Vu20} and \cite{WZ}. 

Since $W^*(X)$ is a Banach space, it has  a natural  functional capacity $\tc(\cdot)$ defined by its norm, which was introduced in \cite{Vigny}. For a Borel subset $E\subset X$, 
\[\label{eq:intro-cap-defn}\notag
	{\tt c}(E) = \inf 
	\left\{ \| v \|_*^2 \;\rvert\;  v\in \Kc(E)
	\right\},
\] 
where
\[\label{eq:intro-ke-cap}\notag
 	\Kc(E) = \left\{ v \in W^*(X) \rv \{v \leq  -1\}^o \supset E \text{ and } v\leq 0 \right\}.
\]
The inclusion $ \{v \leq  -1\}^o \supset E$ means $v\leq -1$ a.e. in a neighborhood of $E$. In fact, this is a capacity in Choquet's sense \cite[Theorem~30]{Vigny}. It should be pointed out that the functional capacity of  with respect to the Sobolev norm is not strong enough to study properties of functions in $W^*(X)$. The reason is that there are bounded and unbounded $\om$-psh functions belong to $W^*(X)$ while the Sobolev capacity cannot be used to characterize neither quasicontinuity of such functions nor pluripolar sets.

The main result of the paper is the following comparison between the global Alexander-Taylor capacity and the functional capacity of $W^*(X)$ above.

\begin{thm} \label{thm:AT-c-intro} There exists a constant $A>0$ such that for every compact subset  $K\subset X$, 
$$	\exp\left( - A/ \tc(K) \right) \leq T_\om(K)   \leq  \exp\left( - 1 / (4n)^\frac{1}{n}   [\tc(K)]^\frac{1}{n}\right).
$$
\end{thm}
Since both $\tc(\cdot)$ and $T_\om(\cdot)$ are Choquet capacities,  the inequalities hold for all Borel sets. Especially, 
they keep the spirit of the ones in  \cite[Lemmas~12.2, 12.3]{GZ-book} where $T_\om(K)$ is compared with the global Bedford-Taylor capacity $cap_\om(K)$. Namely, in the first inequality the constant $A$ is equivalent to the one of \cite[Lemma~12.3]{GZ-book}, up to a numerical constant. In the second one, it is significant that the constant on the right hand side is purely a dimensional constant which is similar to ones in  the local setting \cite{AT84} and global setting \cite[Lemma~12.1]{GZ-book}.  Regarding to the exponents in the inequalities, Remark~\ref{rmk:sharp-ineq} shows that they are sharp.

We expect that this inequality will be a bridge so that the tools from complex Sobolev spaces can be used to study open problems in pluripotential theory and  complex Monge-Amp\`ere equations. Some evidences have been appeared very recently in  \cite{DoN25}, \cite{DKN}, \cite{VV24} and \cite{Vu24, Vu26}.  We give here one more application to the equation with a background form $\om$ which is  {\em semi-positive} and {\em big}. This result  was implicitly contained in \cite[Theorem~3.6]{DKN}.

\begin{cor} \label{cor:ma-intro} Let  $\mu$ be a probability measure on $X$. Assume $\mu$ is $(W^*(X), \log^p)$-continuous with $p>n+1$, i.e., there is constants $c>0$ such that
$$
	\left| \int_X f d\mu  \right| \leq c \left[ 1+ |\log \|f\|_{L^1(X)}|\right]^{-p} 
$$ for every $f \in W^*(X) \cap C^0(X)$ and $\|f\|_*\leq 1$. Then, there exists a bounded $\om$-psh solution $u$ to  $(\om + dd^c u)^n = \mu.$
\end{cor}

\bigskip

{\em Organization.} In Section~\ref{sec:cap} we recall basic facts of pluripotential theory and complex Sobolev spaces. In particular, we give another proof of existence of quasi-continuous representative for bounded functions in $W^*(X)$. Section~\ref{sec:comparisons} devotes to prove the main theorem and corollary.

\bigskip

{\em Acknowledgement.}  A part of the paper was completed while the first author was visiting University of Engineering and Technology (UET) in Hanoi in December 2025. He wish to thank Pham Hoang Hiep for a warm invitation.  The paper was written while the first author  visited the Center for Complex Geometry (Daejeon). He would like to thank Jun-Muk Hwang and Yongnam Lee for their kind support and exceptional hospitality. He is also grateful to the institution for providing perfect working conditions.  The research of the second author is supported by the ministry-level project B2025-SPH-02.

\section{A functional capacity in complex Sobolev spaces} 
\label{sec:cap}
In this section we recall some  basic facts about the complex Sobolev space and pluripotential theory.  Let $(X,\al)$ be a compact K\"ahler manifold of dimension $n$. Let $\om\geq 0$ be a semi-positive closed $(1,1)$-form whose volume  is normalized by 
\[\label{eq:norm-om} \int_X \om^n =1.\]
Clearly, we may assume that 
\[\label{eq:al-om}
	\om \leq \al. 
\]
Let  $W^{1,2}(X,\bR)$ denote the Sobolev space of functions in $L^2(X,\bR)$ whose first  order (weak) derivatives are $L^2$-integrable. For $f\in W^{1,2}(X,\bR),$ let $\Ga_f$  denote the set of all positive closed $(1,1)$-currents $T$ satisfying
\[\label{eq:w-ineq}
	df \wed d^c f  \leq T \quad\text{weakly in } X.
\]
Here we use the normalization 
$$d^c = \frac{\ii}{2\pi} (\dbar -\d), \quad dd^c = \frac{\ii}{\pi}\d\dbar.
$$
\begin{defn}[\cite{DS}]
A function $f\in W^{1,2}(X,\bR)$ belongs to the complex Sobolev space $W^*(X)$ if the set $\Ga_f$ associated to $f$ is non-empty.
\end{defn}
The total mass of a closed positive (1,1)-current $T$ on $X$ is given by
$$
	\|T\| = \int_X T \wed\al^{n-1}.
$$
By the compactness of $X$ and closedness of $T$, $\|T\|$ is always finite. Basic functional properties of $W^*(X)$ have been studied by Vigny \cite{Vigny}. In particular, 
he showed  that this is a Banach space with the $W^*$-norm $\|\cdot\|_*$ defined by
\[\label{eq:w-norm}
	\|f\|_*^2 = \|f\|_{L^2(X)}^2 + \min\{\|T\| : T \in \Ga_f\},
\]
where the minimum is attained thanks to compactness of closed positive $(1,1)$-currents. Though there may be more than one minimizer.

The basic examples of functions in the space are Lipschitz functions. But there are many more. 
It is easy to check that this space contains all bounded quasi-plurisubharmonic (quasi-psh) functions. Hence, for $u, v \in PSH(X,\om) \cap L^\infty(X)$,
$$
	u \pm v \in W^*(X).
$$
We refer the readers to \cite{Vigny} for more examples of functions in $W^*(X)$, which contains unbounded quasi-psh functions.

Although the continuous functions are not dense in $W^*(X)$ with respect to the strong topology of $W^*$-norm \cite[Proposition~7]{Vigny}, however, 
the following compactness \cite[Proposition~4]{Vigny} of $W^*$-norm will be enough for many applications.
\begin{lem} \label{lem:w-compactness}  Let $\{f_j \}_{j\geq 1} \subset W^*(X)$ be such that $\|f_j \|_* \leq A$ for every $j\geq 1$. There exists a subsequence $\{f_{j_k} \}_{j_k \geq 1}$ converging weakly in $W^{1,2}(X,\bR)$ to $f \in W^*(X)$  and $\|f\|_* \leq \liminf_{j\to \infty} \|f_j\|_*$.
\end{lem}

Next, we recall a useful fact whose  proof  is derived easily from \cite[Theorem~2.35]{KLV-book}. 
\begin{prop}  \label{prop:weak-conv-cp} Let $1<p<\infty$ and $E \subset \Om$ be a Borel set.  Let $\{f_j\} \subset L^p(\Om, \bR)$ and $\{g_j\} \subset L^p(\Om,\bR)$ be such that  $f_j \to f \in L^p(\Om,\bR)$ weakly and $g_j \to g\in L^p(\Om,\bR)$ weakly. 
If $|f_j(x)| \leq g_j(x)$ for a.e. $x\in E$, then $|f(x)| \leq g(x)$ a.e. $x\in E$. 
\end{prop}

For a Borel subset $E \subset X$, 
$$
	\Kc(E) = \{v \in W^*(X): \{v \leq -1\}^o \supset  E, v \leq 0 \text{ on }X \}.
$$
Clearly, if $E = G$ is an open subset, then 
$$
	\Kc(G) = \{v \in W^*(X): v \leq -1 \quad \text{a.e. on } G, \; v \leq 0 \text{ on }X \}.
$$
Following Vigny \cite{Vigny}, we define the (global) functional capacity 
\[\label{eq:w-cap}
	\tc(E) = \inf \{ \|v\|_*^2: \; v \in \Kc(E)\}.
\]
Observe that  if $v\in \Kc(E)$, then $\max\{v,-1\} \in \Kc(E)$. An elementary computation shows that $\| \max\{v,-1\} \|_* \leq \|v\|_*$ (see e.g., \cite{Vigny} or \cite{DMV}). Hence, 
in the definition of  $\tc(\cdot)$ it is enough to take $v$ in a smaller family $\Kc'(E)$ defined by
$$
	\Kc'(E) = \left\{ v\in \Kc(E): -1 \leq v \leq 0\right\}.
$$
In other words,
\[\label{eq:w-cap-use}
	\tc(E) = \inf \{ \|v\|_*^2: \; v \in \Kc'(E)\}, 
\]
The following basic properties of the capacity were obtained in \cite[Proposition~27]{Vigny}, whose proofs  are followed from classical arguments (see e.g. \cite{AH96}).

\begin{prop} \label{prop:w-cap-basic-p} The set function on Borel sets $E \mapsto \tc(E)$ satisfies:
\begin{itemize}
\item
[(a)] for Borel sets $E \subset F \subset X$, $\tc(E) \leq \tc (F)$.
\item
[(b)] For a sequence of Borel sets $\{E_i\}_{i\geq 1}$ in $X$, 
$
	\tc (\cup_{i\geq 1} E_i) \leq \sum_{i\geq 1} \tc (E_i).
$
\item[(c)] For $E \subset X$, $\tc(E) \leq \int_X \om^n =1.$
\item
[(d)] For decreasing sequence of compact sets $K_1 \supset K_2 \supset \cdots$ and $K= \cap_{i\geq 1} K_i$, then 
$\tc(K)  = \lim_{i\to \infty} \tc(K_i).$
\end{itemize}
\end{prop}

In order for $\tc(\cdot)$ to be a Choquet capacity (Theorem~\ref{thm:Vigny}), then we need the limit properties under decreasing of compact sets and under increasing of Borel sets. The former one is already given in (d).  However, the latter limit property is the most difficult one to prove for $\tc(\cdot)$ (see \cite[Theorem~30]{Vigny}).

The outer regularity is rather easy so we include its proof.
\begin{lem}[outer regularity] \label{lem:outer-reg} For a Borel subset $E\subset \Om$,
\[\notag	\tc(E) = \inf\{ \tc(G) : E \subset G,\; G \text{ is open}\}.
\]
\end{lem}

\begin{proof} By the monotonicity in Proposition~\ref{prop:w-cap-basic-p}-(a), 
$$
	\tc(E) \leq \inf\{ \tc(G) : E \subset G,\; G \text{ is open}\}.
$$
To prove the reverse inequality we may assume $\Kc(E)$ is non empty. Let $\veps>0$. There exist $v\in \Kc(E)$ with 
$
	\|v\|_*^2 \leq \tc(E) +\veps
$
and  an open set $G_\veps$ such that $v(x)\leq - 1$ a.e. $x\in G_\veps$. Therefore, 
$$
	\inf\{ \tc(G) : E \subset G,\; G \text{ is open}\} \leq \tc(G_\veps) \leq \|v\|_*^2 \leq \tc(E) +\veps.
$$
The lemma follows by letting $\veps\to 0$.
\end{proof}

\subsection{Local and global definitions of capacities} 

The global Bedford-Taylor capacity was introduced in \cite{Ko03} and studied in detail in \cite{GZ05, GZ-book}.  For a Borel subset $E\subset X$, one defines 
\[\label{eq:cap-om}
	cap_\om (E) = \sup \left\{\int_E (\om + dd^c v)^n : v\in PSH(X, \om), -1\leq v \leq 0 \right\}.
\]
Then,  we have a formula \cite[Theorem~9.15]{GZ-book} realizing the capacity via extremal functions. Precisely, for a compact subset $K\subset X$,
\[\label{eq:cap-id}
	cap_\om (K) = \int_K (\om+ dd^c h_K^*)^n = \int_X - h_K^* (\om+dd^c h_K^*),
\]
where $h_K^*$ denotes the upper semi-contitinuous regularization of 
\[\label{eq:ext-h-formula}
	h_K(z) = \sup\{ v(z): v \in PSH(X,\om), v_{|_K} \leq -1, v \leq 0\}.
\]
As in local case \cite{BT82}, this global capacity is a Choquet capacity and it characterizes locally/globally pluripolar sets \cite[Theorem~9.17]{GZ-book}. Namely, a Borel set $E\subset X$ is pluripolar if and only if
\[\label{eq:pluripolar-cap}	
	cap_\om(E) =0 \quad \Leftrightarrow \quad cap_\al (E) =0,
\]
where $cap_\al(E)$ is considered with respect to  $PSH(X,\al)$ and $\al$   in  \eqref{eq:cap-om}.

We consider  yet another capacity defined via local covering and the local relative capacity of Bedford and Taylor \cite{BT82}  as follows. Let $\Om \subset\subset \bC^n$ be an open subset. Then,  for a Borel set $E\subset \Om$, 
\[
	cap(E, \Om) = \sup\left\{ \int_\Om (dd^c v)^n : v\in PSH(\Om), -1 \leq v \leq 0\right\}.
\]
Consider two finite coverings $\{D_\ell\}_{\ell}$ and  $\{\Om_\ell\}_{\ell}$ such that $D_\ell \subset \subset \Om_\ell$ for each $1\leq \ell \leq N$ and each holomorphic chart $\Om_\ell$ is a smooth strictly pseudoconvex domain. The  quantity $cap_{\rm BT}(E)$ is given by 
$$
	cap_{\rm BT}(E) = \sum_{\ell=1}^N cap(E \cap D_\ell, \Om_\ell).
$$
By the observation of Ko\l odziej \cite[Page 670]{Ko03}  (see also \cite[Proposition~9.8]{GZ-book}), these two capacities $cap_\al(\cdot)$ and $cap_{\rm BT}(\cdot)$ are equivalent to each other. More precisely,  there exists a constant $A_0$ depending only on the covering and $\al$ such that
\[\label{eq:cap-loc-vs-glob}
	\frac{1}{A_0}  cap_{\rm BT}(E) \leq cap_\al(E)  \leq A_0 cap_{\rm BT}(E).
\]
 When the background metric is K\"ahler,
we have a comparison between their global counterparts from \cite[Theorem 28]{Vigny} and \cite[Proposition~5.1]{DKN} ({\em c.f.} Remark~\ref{rmk:DKN})

\begin{lem}\label{lem:cap-c} There exists a constant $A>0$ such that for every Borel set $E\subset X$,
$$
	\frac{1}{A} cap_\al(E) \leq  \tc(E) \leq A [cap_\al(E)]^\frac{1}{n}.
$$
Consequently, if $\tc (E) =0$, then $E$ is pluripolar.
\end{lem}

\subsection{Quasi-continuous representative} Since the functions in $W^*(X)$ are defined up to a set of Lebesgue measure zero,
we need to study further their fine properties for applications. To make it precise we define the following terminologies.

\begin{defn} \label{defn:c-quasi-c}
\begin{itemize}
\item
[(a)] A Borel function  $f$ is called quasi-continuous (q.c) in $X$ if there exists for every $\veps> 0$ an open set $G \subset X$ such that $\tc (G)<\veps$ and the restriction $f_{|_{X\setminus G}}$  is continuous  (with respect to the induced topology on $X\setminus G$).
\item
[(b)] A statement is said to hold {\em quasi-everywhere} (q.e.) if there exists a Borel set $P$ of capacity zero, i.e., $\tc (P) =0$, such that the statement is true for every $x\notin P$. 
\end{itemize}
\end{defn}

\begin{remark} By Lemma~\ref{lem:cap-c}  the quasi-continuity and quasi-everywhere properties with respect to $\tc (\cdot)$ are equivalent to the classical ones  in pluripotential theory \cite{BT82, BT87} and \cite{GZ-book}.
\end{remark}

The following notion helps us to get a good representative in each class of functions.

\begin{defn}
Let $f,g$ be Borel measurable functions on $X$. Then,  $g$ is said to be a quasi-continuous modification of $f$ if $g$ is quasi-continuous and $g= f$ a.e. We denote $g$ by $\wt f$ in this case.
\end{defn}

We also need a classical result in potential theory \cite[Lemma~2.14]{FOT94} (see also \cite[Propositon~21]{Vigny}). 
\begin{prop} \label{prop:ae-qe}  Let $D\subset X$ be an open set and $f$ be a quasi-continuous function on $D$. If $f\geq 0$ a.e. on $D$, then $f\geq 0$  quasi-everywhere on  $D$.
\end{prop}

Thanks to this result, two quasi-continuous modifications are equal q.e. Hence, they are equal outside a pluripolar set by Lemma~\ref{lem:cap-c}. The existence of a quasi-continuous modification with respect to the Bedford-Taylor capacity was obtained by  Vigny \cite[Theorem~22]{Vigny}.  We state below the case of bounded functions which is enough for our applications.  By using the local result of Dinh, Marinescu and Vu \cite{DMV} we also provide  a different proof for the reader's convenience.

If $f \in W^*(X)$, then for every holomorphic coordinate chart $\Om \subset \subset X$, $f \in W^*(\Om)$ the local complex Sobolev space (see \cite{DMV} and \cite{Vigny}). The definition of $W^*(\Om)$ is very similar. Namely,  let $\be = dd^c |z|^2$ be the standard K\"ahler metric in $\bC^n$. For $f\in W^{1,2}(\Om,\bR),$ let denote $\Ga_f(\Om)$ be the set of all closed positive $(1,1)$-currents $T$ in $\Om$ satisfying
 \[\label{eq:w-ineq-loc}
	df \wed d^c f  \leq T \quad\text{weakly in } \Om.
\]

\begin{defn}\label{defn:w-local}
A function $f\in W^{1,2}(\Om,\bR)$  belongs $W^*(\Om)$ if   there exists $T \in \Ga_f (\Om)$ such that $\|T\|_\Om = \int_\Om T \wed \be <+\infty$. \end{defn}
The local $W^*(\Om)$-norm is given by
\[\label{eq:w-norm-loc}
	\|f\|_{*,\Om}^2 = \|f\|_{L^2(\Om)}^2 + \min\{\|T\|_\Om : T \in \Ga_f(\Om) \text{ and } \|T\|_\Om <+\infty\},
\]
It was shown in \cite[Proposition~1]{Vigny} that $(W^*(\Om), \|\cdot\|_{*,\Om})$ is also a Banach space.

\begin{thm} \label{thm:quasi-mod-bdd} Let $f\in W^*(X) \cap L^\infty(X)$. Then, there exists a quasi-continuous modification $\wt f \in W^*(X) \cap L^\infty(X)$. Moreover, if $g$ is another quasi-continuous modification of $f$, then $g =\wt f$ outside a pluripolar set.
\end{thm}

\begin{proof} By Lemmma~\ref{lem:cap-c} it is enough to prove the quasi-continuity with respect to $cap_\al(\cdot)$. 
 Let $\chi$ be a cut-off function on $X$.  Since $f$ is bounded, it follows that $\chi f$ belongs to $W^*(X)$.  By partition of unity and subadditivity of $cap_\al (\cdot)$ (see e.g. \cite{GZ-book}), it is enough to assume $f$ has a compact support in  a holomorphic coordinate ball $\Om\subset \subset X$. Clearly, $f$ restricted to $\Om$ belongs to  the local complex Sobolev space $W^*(\Om)$. By \cite[Theorem~2.10]{DMV}, there exists a quasi-continuous modification $\wt f \in W^*(\Om)$ with respect to $cap(\cdot, \Om)$ whose support is compact in $\Om$. The equivalence \eqref{eq:cap-loc-vs-glob} implies that this function is a desired one. 

Let us verify the second statement. Assume $\wt f$ is quasi-continuous and $g= \wt f$ a.e. Then, it follows from Proposition~\ref{prop:ae-qe} that $g = \wt f $ on $X\setminus F$ for  a Borel subset $F\subset X$ satisfying $\tc (F) =0$. Hence, $cap_\al(F)=0$ by  Lemma~\ref{lem:cap-c}. Now it follows from \eqref{eq:cap-loc-vs-glob}  and \cite{BT82} that $F$ is locally pluripolar. Finally, by the characterization \cite[Theorem~12.5]{GZ-book},  $F$ is globally pluripolar, i.e., there exists a function $u\in PSH(X,\om)$ such that $F \subset \{u=-\infty\}$.
\end{proof}

A crucial technical tool to study a capacity is its extremal function. An analogous identity of \eqref{eq:cap-id} for $\tc(\cdot)$ was obtained in \cite[Corollary~31]{Vigny}. We prove here also  the uniqueness which answers the question given there for open sets.

\begin{lem} \label{lem:extr-funciton-o} Assume $E\subset X$ is an open subset. There is a unique $\Phi_E \in W^*(X)$ such that 
$\tc (E) = \|\Phi_E\|_*^2$. Moreover, $-1\leq \Phi_E \leq 0$ a.e. and $\Phi_E=-1$ a.e. on $E$.
\end{lem}

\begin{proof} We first prove the existence. Let $\{u_j\}_{j\geq 1}$ be a sequence in $\Kc(E)$ such that $\tc(E) = \lim_{j\to \infty} \|u_j\|_*^2$. It follows from Lemma~\ref{lem:w-compactness} that $u_j$ converges weakly in $W^{1,2}(X)$ to  $u \in W^*(X)$ and $\|u\|_* \leq \liminf_j \|u_j\|_*$. Since $u_j \to u$ weakly in $L^2(X)$,  it follows from Proposition~\ref{prop:weak-conv-cp} that $u =-1$ a.e. in $E$ and $u\leq 0$ in $X$. Note that  the proof of that proposition used only the properties of Lebesgue points of functions in $L^2(X)$. Since $E$ is open,  $u\in \Kc(E)$ (this is the sole place that the openness of $E$ is used). By definition of  $\tc(E),$ we get $\tc(E) = \|u\|_*^2$. Here, we can choose $\Phi_E = \wt u$ which is quasi-continuous and $\Phi_E =-1$ q.e. on $E$.

Next,  we show the uniqueness. Suppose $u_1$ and $u_2$ are such two functions. Let $T_1$ and $T_2$ are the corresponding closed positive currents such that
$$
	\tc(E) = \|u_1\|_{L^2}^2 + \|T_1\| = \|u_2\|_{L^2}^2 + \|T_2\|.
$$  
Define $v = (u_1+u_2)/2$. Then, 
$
	dv \wed d^c v \leq (T_1 +T_2)/2 =:T.
$
So, $v\in \Kc(E)$ and 
$$\begin{aligned}
	\tc(E) &\leq \|v\|_{L^2}^2 + \|T\| \\&=\frac{1}{4} (\|u_1\|_{L^2}^2 + \| u_2\|_{L^2}^2) + \frac{1}{2} \|u_1u_2\|_{L^1} + \frac{1}{2} (\|T_1\|+ \|T_2\|). 
\end{aligned}$$
This implies that 
$$
	2\int_X u_1 u_2  \al^n \geq \int_X u_1^2 \;\al^n + \int_X u_2^2 \; \al^n.
$$
Thus, $u_1 = u_2$ a.e. If $u_1, u_2$ are quasi-continuous, then  $u_1=u_2$ q.e. by Proposition~\ref{prop:ae-qe}.
\end{proof}

Now we have a refined definition of capacity.
 
\begin{lem} \label{lem:cap-refine} Let $E \subset X$ be a Borel set. Then,
$$
	\tc(E) = \inf\{\|v\|_*^2 : v \in 
	\wt\Kc(E)\},
$$
where 
$$
 	\wt\Kc(E) = \left\{ v \in W^*(X) \;\rvert\; \wt v \leq  -1 \text{ q.e. on $E$ and } \wt v\leq 0 \right\}.
$$
\end{lem}

\begin{proof} If $v\leq -1$ a.e. on a neighborhood of $E$, then $\wt v \leq -1$ q.e. by Propostion~\ref{prop:ae-qe}. So,  $\Kc(E) \subset \wt\Kc(E)$ and  $\tc (E)$ is greater than the infimum on the right hand side.
To show the reverse inequality let us consider $h\in \wt\Kc_E$ and fix a quasi-continuous modification $\wt h$ of $h$ (Theorem~\ref{thm:quasi-mod-bdd}). Given $\veps>0,$ we choose an open set $G:= G_\veps$ such that $\tc (G)<\veps$ and $\wt h$ restricted to $X\setminus G$ is continuous and $\wt h \leq -1$ on $E \cap (X\setminus G)$. Denote  by $u$ the function $\Phi_{G}$ of  the open set $G$ in Lemma~\ref{lem:extr-funciton-o}. The set
$$
	D_\veps = \{x\in X \setminus G: \wt h < -1+\veps \} \cup G
$$
is open and $E\subset D_\veps$. Moreover, $h + u \leq -1 + \veps$ a.e. on $D_\veps$ and it is negative in $X$.
Therefore,
$$\begin{aligned}
	\tc (E) &\leq \tc (D_\veps) \leq \frac{\|h+u\|_*^2}{(1-\veps)^2}  \\& \leq \frac{ (\|h\|_* +\|u\|_*)^2}{(1-\veps)^2} \\ &\leq \frac{ (\|h\|_* + \sqrt{\veps})^2}{(1-\veps)^2}.
\end{aligned}$$
By letting $\veps \searrow 0$, $\tc(E) \leq \|h\|_*^2$. Since $h$ is arbitrary, the proof is completed.
\end{proof}

\begin{remark}
 Vigny \cite[Theoerem~30]{Vigny} proved that  this functional capacity is a Choquet capacity.
As a consequence, $\tc (\cdot)$ is also inner regular, i.e.,
\[\notag\label{eq:inner-reg}
	\tc(E)= \sup\left\{ \tc(K) : K \subset E, \; K\text{ is compact}\right\}.
\]
\end{remark}
\subsection{The Alexander-Taylor capacity} Let $K\subset X$ be a compact subset. The global Siciak-Zaharjuta extremal function associated with $K$ is given by 
\[\label{eq:SZ}
	V_K(z) = \sup\{v(z): v\in PSH(X,\om), v\leq 0 \text{ on K}\}.
\]
Denote by $V_K^*$  the upper semi-continuous regularization of $V_K$. Recall that the Alexander-Taylor capacity is given by 
\[\notag\label{eq:AT-cap}
	T_\om (K) = \exp \left( - \sup_X V_K^* \right).
\]
This is also a Choquet capacity and $K$ is pluripolar if and only if $T_\om(K) =0$. Equivalently, 
\[ \label{eq:SZ-polar} K \text{ is pluripolar } \Leftrightarrow 
\sup_X V_K^* =+\infty.
\] 
Similar to the local case \cite{AT84}, it follows from \cite[Lemma~12.2, Lemma~12.3]{GZ-book} that the global Alexander-Taylor inequality is well comparable with the global Bedford-Taylor capacity. Namely, there exists a uniform constant $A_1>0$ such that for every compact sets $K \subset X$,
\[\label{eq:global-AT}
\exp\left(\frac{- A_1}{cap_\om(K)} \right)	\leq T_\om(K) \leq  \exp\left( \frac{-1}{ [cap_\om(K)]^\frac{1}{n}} \right).
\]
Notice that since  $\int_X \om^n =1,$ we always have $cap_\om(K) \leq cap_\om(X) = 1$.

\section{Comparison of capacities}
\label{sec:comparisons}

In this section we obtain comparisons between the functional capacity $\tc(\cdot)$ and two other classical capacities in pluripotential theory, the Bedford-Taylor capacity and  the Alexander-Taylor capacity. We emphasize that the arguments are global in nature. This is why  it works for $\om,$ where $\om$ is merely a {\em semi-positive and  big} form.

We start with  a technical result for estimating to the integral of functions in $W^*(X)$ against Monge-Amp\`ere measures of bounded psh functions. 
Let $\phi \in PSH(X,\om)$ be such that $-1\leq \phi \leq 0$. Write 
$$	
	\om_\phi := \om + dd^c\phi.
$$For $f \in W^*(X) \cap L^\infty(X)$, we define the integral 
\[\label{eq:defn-int}
\int_X f \; \om_\phi^n := \int_X \wt f \; \om_\phi^n.
\]
This definition is well-defined. In fact, 
if $f=g$ a.e., then $\wt f = \wt g$ q.e.  by Proposition~\ref{prop:ae-qe}. Since the Monge-Amp\`ere measure $\om_\phi^n$ does not charge on pluripolar sets,  the integral  on the right hand side does not depends on the quasi-continuous modification we chose. 

\begin{lem} \label{lem:IBP-ineq}Let $h \in PSH(X,\om)$ be such that $-1 \leq h \leq 0$. Let $v \in W^*(X)$ satisfy  $-1\leq v \leq 0$ and $dv \wed d^c v \leq T$ for a closed positive $(1,1)$-current $T$. Then, 
$$
\begin{aligned}
	\int_X v^2 \om_h^n 
&\leq 	\int_X v^2 \om_h^{n-1}\wed \om 
+ 2 \int_X T \wed\om^{n-1} \\ 
&\quad 	+ \frac{1}{2} \int_X -h \om_h^n + \frac{1}{2} \int_X h\om_h^{n-1}\wed \om.
\end{aligned}$$
\end{lem}

\begin{proof}  The proof goes through  by integration by parts.  We first prove it for  $h$ and $v$ being smooth and then the general case via approximation. The subtle point is that we may not find a good smooth sequence in $W^*(X)$ which converges strongly to $v$ in $W^*(X)$. Fortunately, we can resolve this by using the quasi-continuity of $v$.   It follows from Theorem~\ref{thm:quasi-mod-bdd}  and Lemma~\ref{lem:cap-c} that $v$ admits a quasi-continuous representation $\wt v$ with respect to the capacity $cap_\al(\cdot)$. By definition~\eqref{eq:defn-int}, we can assume $v$ is quasi-continuous itself. Moreover, by considering $\om:=\om + \de \al$ and then letting $\de \to 0^+,$ we may assume that $\om$ is a K\"ahler form. 

{\bf Step 1a}: Assume both $v$ and $h$ are smooth. Compute
\[\label{eq:1st-ineq-id}
	\int_X v^2 \om_h^n = \int_X v^2 \om_h^{n-1} \wed\om + \int_X v^2 dd^c h \wed \om_h^{n-1}.
\]
By the Stokes theorem and Cauchy-Schwarz inequality, we have
\[\label{eq:1st-ineq-IBP}\begin{aligned}
	\int_X v^2 dd^c h \wed \om_h^{n-1} 
&= 2\int_X (-v) dv\wed d^c h \wed \om^{n-1}_h \\
&\leq 2 \left( \int_X dv \wed d^c v \wed \om_h^{n-1}\right)^\frac{1}{2} \left( \int_X v^2dh \wed d^c h \wed \om_h^{n-1}\right)^\frac{1}{2} \\
&\leq 2 \int_X dv \wed d^c v\wed \om_h^{n-1} + \frac{1}{2} \int_X d h \wed d^c h \wed \om_h^{n-1},
\end{aligned}\]
where in the last inequality we  used the fact that $0\leq v^2 \leq 1$.
Using the Stokes theorem again,
$$\begin{aligned}
	\int_X dh \wed d^c h \wed \om_h^{n-1} 
&= \int - h dd^c h\wed\om_h^{n-1} \\
&= \int_X -h \om_h^n + \int_X h \om_h^{n-1} \wed \om.
\end{aligned}$$
Combining this, \eqref{eq:1st-ineq-id} and \eqref{eq:1st-ineq-IBP}, we get the conclusion of the lemma.

{\bf Step 1b}: Assume only $h$ is smooth. Let $\{v_j\}_{j\geq 1}$, $0\leq v_j \leq 1$, be a smooth sequence which converges strongly to $v$ in the usual Sobolev norm of $W^{1,2}(X,\bR)$. It follows from \eqref{eq:1st-ineq-IBP}  that
$$
	\int_X v_j^2 dd^c h \wed \om_h^{n-1} \leq 2 \int_X dv_j \wed d^c v_j \wed \om_h^{n-1}+  \frac{1}{2} \int_X dh \wed d^c h \wed \om_h^{n-1}.
$$
Letting $j\to \infty,$ we conclude that  \eqref{eq:1st-ineq-IBP} holds for $v\in W^*(X)$ and $h$ is smooth. Combining it with the assumption $dv\wed d^c v \leq T,$ we get
$$
	\int_X v^2 \om_h^n  \leq \int_X v^2 \om_h^{n-1}\wed \om + 2 \int T \wed \om_h^{n-1} +  \frac{1}{2} \int_X dh \wed d^c h \wed \om_h^{n-1}.
$$
By  integration by parts and Stokes' theorem,
$$\begin{aligned}
&	\int_X dh\wed d^c h \wed \om_h^{n-1} = \int_X -h \om_h^n + \int_X h \om_h^{n-1} \wed \om, \\
&	\int_X T\wed \om_h^{n-1} = \int_X T\wed \om^{n-1}.
\end{aligned}$$
 Finally, we conclude 
\[ \label{eq:1st-ineq-red}\begin{aligned}
	\int_X v^2 \om_h^n 
&\leq 	\int_X v^2 \om_h^{n-1}\wed \om  	+ 2 \int_X T \wed\om^{n-1} \\ 
&\quad 	+ \frac{1}{2} \int_X -h \om_h^n + \frac{1}{2} \int_X h \om_h^{n-1}\wed \om.
\end{aligned}\]

{\bf Step 2:} Lastly, we remove the smoothness assumption on $h$ in \eqref{eq:1st-ineq-red}. This is where we use the quasi-continuity of $v^2 \in W^*(X)$. Let $\{h_j\}_{j\geq 1}$ be a sequence of smooth $\om$-psh functions decreasing to $h$ (see. e.g., \cite{BK07}). 
By Step 1b, the inequality \eqref{eq:1st-ineq-red} holds for $h_j$. By letting  $j\to \infty$  and  invoking the weak convergence theorem  \cite[Theorem~3.2]{BT87} or \cite[Theorem~4.26]{GZ-book}, we conclude that \eqref{eq:1st-ineq-red} holds for a  general $h$. 
\end{proof}

\begin{lem}\label{lem:star-norm-SZ} Let $E\subset X$ be a compact subset. Let $v \in PSH(X,\om)$ be such that $v \leq 0$ on $E$ and $M:=\sup_X v \geq 1$. Define $\phi = (v-M-1)/M$. Then, $\phi\in \Kc(E)$ satisfies
$$
	\|\phi\|_*^2 \leq \frac{A_0}{M},
$$
where $A_0$ is a uniform constant  depending only on $X,\om$.
Consequently, $\tc (E) =0$ for  a compact pluripolar set $E$. 
\end{lem}

\begin{proof}  We have
$$
 	-1 \leq u:= \frac{v- M}{M}  \leq 0.
$$
Hence $\phi = u - 1/M \in \Kc(E)$.  Next we compute $W^*$-norm of $\phi$.  The normalization $\sup_X (v-M)  =0$ and $v\in PSH(X,\om)$ imply  a well-known fact that 
$$
	\int_X |u| \al^n = \int_X \frac{M-v}{M} \al^n\leq \frac{A_1}{M}.
$$
for a uniform constant $A_1>0$. 
Since $|\phi| \leq 2$, we get
\[\label{eq:L2-norm-v}
\|\phi\|_{L^2(X)}^2 =	\int_X (u -1/M)^2 \al^n \leq \int_X -2(u-1/M) \al^n \leq  \frac{A_2}{M},
\]
where $A_2 =  2A_1 +2 \int_X\al^n.$
 Moreover, as $1+u \in PSH(X,\om/M)$ and $0\leq 1+ u \leq 1$, we have
$$
	T := \frac{\om}{M} + \frac{1}{2} dd^c u^2 + dd^c u \geq du \wed d^c u = d \phi \wed d^c \phi. 
$$
It follows from \eqref{eq:L2-norm-v} and the Stokes theorem that 
\[\notag
	\|\phi\|_*^2 \leq \|\phi\|_{L^2(X)}^2 + \|T\| \leq  \frac{A_2}{M} + \frac{1}{M}  \int_X \om \wed \al^{n-1} =:\frac{A_0}{M}.
\]
This completes the proof of the first statement.

Now assume $E$ is a compact pluripolar set. By \eqref{eq:SZ-polar}, $\sup_X V_E^*= +\infty$, where $V_E^*$ is its associated Siciak-Zaharjuta extremal function defined in \eqref{eq:SZ}. Choquet's lemma yields there exists a sequence $\{v_j\}_{j\geq 1} \subset PSH(X,\om)$ satisfying $v_j \leq 0$ on $E$ and 
$$
	M_j:=\sup_X v_j \to +\infty 
$$ 
as $j\to \infty$. Applying the first part for $\phi_j = (v_j - M_j - 1)/M_j,$ we get  $\|\phi_j\|_*^2 \leq A_0/M_j$. Since $\phi_j \in \Kc(E)$,
it follows that $\tc (E) \leq A_0/M_j$. Letting $j\to +\infty,$ the second statement follows.
\end{proof}

The following result is the key estimate to prove Theorem~\ref{thm:AT-c-intro}. It provides a precise quantitative comparison between  $\tc(\cdot)$ and $cap_\om(\cdot)$. 

\begin{thm}\label{thm:Vigny} There exists a constant $A>0$ such that for every compact   subset $E\subset X$, 
$$
	\frac{1}{4n} cap_\om(E) \leq \tc(E) \leq A \left[cap_\om(E) \right]^\frac{1}{n}.
$$
\end{thm}

\begin{proof}  Let $E\subset X$ be compact. 
Let $h_E^*$ be the relative extremal function of $E$ defined in \eqref{eq:ext-h-formula}.
Then, $0\leq h_E^* \leq 1$ and $h_E^* \in PSH(X,\om)$. 
Next, let $V_E^*$ be the global extremal function for $E$ as in \eqref{eq:SZ}. If $E$ is pluripolar, then $cap_\al(E)=0$ by \cite[Corollary~9.9, Theorem~9.17]{GZ-book}, so is $cap_\om(E)=0$. Similarly, $\tc(E)=0$ by Lemma~\ref{lem:star-norm-SZ}. Hence, both inequalities hold true.  Therefore, from now on we  assume $E$ is non-pluripolar. 
By \eqref{eq:SZ-polar} we have $0\leq V_E^* \in PSH(X,\om)$ and 
$0\leq M:= \sup_X V_E^*<+\infty.$

Let us prove  the second inequality in the theorem.  If $0 \leq M \leq 1$, then $h_E^* = V_E^* -1$. It is easy to see from \eqref{eq:cap-id} that 
$$
	\int_X \om^n = \int_K (\om+ dd^c V_E^*)^n \leq cap_\om(E)  \leq \int_X \om^n.
$$
Hence,
$cap_\om(E) = cap_\om(X)=1$ by normalization \eqref{eq:norm-om}. Hence,  the conclusion  follows immediately from Proposition~\ref{prop:w-cap-basic-p}-(c).

It remains to consider  $M \geq 1$.
Applying Lemma~\ref{lem:star-norm-SZ} for $v= V_E^*$,  
and $\phi = (V_E^* - M-1)/M,$ we obtain 
\[\label{eq:w-norm-AT-cap}
	\|\phi\|_*^2 \leq \frac{A_0}{M}.
\]
Since $\{V_E < V_E^*\}$ is pluripolar,  $V_E^* \leq 0$ q.e. on $E$ and therefore  $\phi\in \wt\Kc(E)$.

The  global Alexander-Taylor inequality \eqref{eq:global-AT} is equivalent to  
\[\label{eq:global-AT-use}
\frac{cap_\om(E)}{A_1} 	\leq \frac{1}{M} \leq  [cap_\om(E)]^\frac{1}{n}
\]
for a uniform constant $A_1>0$. Thus, Lemma~\ref{lem:cap-refine} yields
$$
	\tc(E) \leq \|\phi\|_*^2 \leq  A_0 [cap_\om(E)]^\frac{1}{n}.
$$

We proceed to prove the first inequality of the theorem. 
A weaker version  is already proved in Lemma~\ref{lem:cap-c}. Here we use a global argument to improve significantly the constant and remove K\"ahler condition on $\om$. Recall that $E\subset X$ is compact. Let  $v\in \wt\Kc(E)$ be such that $-1 \leq v \leq 0$. Hence, $\wt v = -1$ q.e. on $E$.  For simplicity, we write  $h:= h_E^*$ the relative extremal function associated to $E$ and write $v$  for $\wt v$ in what follows. 
Since $v = -1$ q.e. on $E$, it follows from \eqref{eq:cap-id} that 
\[\label{eq:rough-ineq} 
	cap_\om(E) =\int_E   \om_h^n \leq \int_X v^2 \om_h^n.
\]
Now we will estimate the integral
$
	\int_X v^2 (\om + dd^c h)^n
$
from above. In fact, let $T$ be a positive closed (1,1)-current realizing the $W^*$-norm for $v$. In particular, 
$$
	dv \wed d^c v \leq T.
$$
By  Lemma~\ref{lem:IBP-ineq},  we obtain
\[\label{eq:1st-ineq-red-gen}\begin{aligned}
	\int_X v^2 \om_h^n 
&\leq 	\int_X v^2 \om_h^{n-1}\wed \om  	+ 2 \int_X T \wed\om^{n-1} \\ 
&\quad 	+ \frac{1}{2} \int_X (-h) \om_h^n + \frac{1}{2} \int_X h\om_h^{n-1}\wed \om.
\end{aligned}\]
Note  $cap_\om(E) = \int_X -h \om_h^n$ by \eqref{eq:cap-id}. Hence, combining  \eqref{eq:rough-ineq} and \eqref{eq:1st-ineq-red-gen}, we get 
\[\label{eq:first-step}
	\frac{1}{2} cap_\om(E) \leq \int_X v^2 \om_h^{n-1} \wed \om + 2 \|T\|_\om + \frac{1}{2} \int_X h\om_h^{n-1} \wed \om,
\]
where we used the notation $\|T\|_\om = \int_X T\wed \om^{n-1}$.
Next, 
$$
	\int_X v^2 \om_h^{n-1} \wed \om =\int_X v^2\om_h^{n-2} \wed \om^2 +  \int_X v^2 dd^c h \wed \om_h^{n-2}\wed \om.
$$
Observe that we  can apply the argument in \eqref{eq:1st-ineq-IBP}  for  $dd^c h\wed \om_h^{n-2}\wed\om$  in the place of $dd^c h \wed \om_h^{n-1}$. This  yields
\[\label{eq:inductive-step}
\begin{aligned}
	\int_X v^2 \om_h^{n-1} \wed \om 
&\leq 	\int_X v^2\om_h^{n-2} \wed \om^2  + 2 \|T\|_\om \\
&\quad + \frac{1}{2} \int_X(-h) \om_h^{n-1} \wed \om + \frac{1}{2}\int_X h \om_h^{n-2}\wed\om^2.
\end{aligned}\]
The important point here is  that  when plugging this estimate into \eqref{eq:first-step} the third integral on the right hand side is cancelled out. Hence, 
 we get
$$
	\frac{1}{2} cap (E) \leq   \int_X v^2\om_h^{n-2}\wed\om^2 + 4 \|T\|_\om + \frac{1}{2} \int_X h\om_h^{n-2}\wed \om^2.
$$
Compared to \eqref{eq:first-step}, we lowered the exponent of $\om_h$ by one after using the integration by parts inequality  in Lemma~\ref{lem:IBP-ineq}.
Repeating the process for $(n-2)$-times more, we finally get
$$\begin{aligned}
	\frac{1}{2} cap_\om(E) 
&\leq  \int_X v^2\om^n  + 2n \|T\|_\om + \frac{1}{2}\int_X h \om^n \\
& \leq  \|v\|_{L^2(X)}^2 + 2n \|T\|\\
&\leq 2n \|v\|_*^2,
\end{aligned}$$
as $h$ is negative, $T$ realizes the $W^*$-norm of $v$ and $\om \leq \al$. Since $v\in \wt\Kc(E)$, $-1\leq v\leq 0$ is arbitrary, we conclude  $cap_\om(E) \leq 4n \;\tc (E)$ by Lemma~\ref{lem:cap-refine}. 
\end{proof}

\begin{remark} \label{rmk:DKN} The second inequality in the theorem was obtained in \cite[Proposition~5.1]{DKN} when $\om=\al$ is a K\"ahler form and $\tc(E)$ was defined by another equivalent norm with $W^*$-norm. More precisely, for $f \in W^*(X)$, one defines
$$
	\| f\|_{*'} = \|f\|_{L^1(X)} + \min\{\|T\|^\frac{1}{2} : T \in \Ga_f\}.	
$$
Clearly, $\|f\|_{*'}^2 \leq 2 \| f\|_*^2$. The other direction follows from the Poincar\'e inequality. However, the sharpness of exponents were not known in \cite{DKN}.  A qualitative version of  the second inequality was also obtained in \cite[Proposition 29]{Vigny} without a specific exponent. 
\end{remark}

We are ready to get the comparison between capacities in the main theorem.

\begin{proof}[Proof of Theorem~\ref{thm:AT-c-intro}] Let $K \subset X$ be a compact subset. It follows from \eqref{eq:global-AT} and Theorem~\ref{thm:Vigny} that we may assume $K$ is non-pluripolar. Let $V_K^*$ be the Siciak-Zaharjuta extremal function associated to $K$ and $M = \sup_X V_K^*$.  By definitions of two capacities, the first inequality is equivalent to 
$$
	\tc(K) \leq A/M
$$
for a uniform constant $A>0$. If $ 0 \leq M <1$, then it is obvious as $\tc(K) \leq 1$ (Proposition~\ref{prop:w-cap-basic-p}-(c)). Otherwise,  it is an immediate consequence of \eqref{eq:w-norm-AT-cap}.
Next, we have from \eqref{eq:global-AT-use} and Theorem~\ref{thm:Vigny} that 
$$
	1/M \leq [cap_\om(E)]^\frac{1}{n} \leq [4n \tc(E)]^\frac{1}{n}.
$$ 
This is equivalent to the second inequality. 
\end{proof}

The optimality of exponents are followed from the following concrete examples.

\begin{remark}\label{rmk:sharp-ineq} The inequalities in Theorem~\ref{thm:AT-c-intro}  as well as in  Theorem~\ref{thm:Vigny} are sharp as far as the exponents are concerned. Indeed,  let us consider $X = \bP^n$ and $\om$ be the Fubini-Study metric. The first inequality in Theorem~\ref{thm:Vigny} and Theorem~\ref{thm:AT-c-intro} implies
\[ \label{eq:AT-sharp}
	\exp\left( - A'/ cap_\om (K)\right) \leq T_\om(K),
\]
where $A' = A/4n$.
 By \cite[Example~9.22]{GZ-book} that the global capacity $T_\om (\cdot)$ and its local version $T_{\bB} (\cdot)$, where $\bB = B(0,1) \subset \bC^n$ is the unit ball, are comparable modulo a uniform constant. Namely,
 $$
 	\frac{1}{\sqrt{2}} T_\om(K) \leq T_{\bB}(K) \leq 2 T_\om(K).
 $$
  Moreover,  \eqref{eq:cap-loc-vs-glob} shows that $cap_\om(E)$ and $cap (E, \Om)$ are equivalent to each other up to a uniform constant for every Borel set $E \subset D \subset \subset \Om$ where $\Om \subset \subset X$ is a strictly pseudoconvex local coordinate chart. Let  $K = \{ (z_1,...,z_n)\in \bC^n: |z_1| \leq \de, |z_j|\leq 1/2, \; j=2,...,n\}$ $(\de<1/2)$ be  a polydisc, which is contained in a local coordinate unit ball.  Alexander and Taylor \cite[Remark~2]{AT84} showed that the exponents in \eqref{eq:AT-sharp} are sharp for such compact sets. So are the ones in the first inequality in Theorem~\ref{thm:Vigny}.
 
Next, the right hand side inequality in \eqref{eq:global-AT} reads
$$	
	T_\om(K) \leq \exp\left( \frac{-1}{ [cap_\om(K)]^\frac{1}{n}} \right)
$$
where the exponents are sharp by the same local comparable reason and \cite[Remark~2]{AT84} as above. In this case the example is given by a small closed ball $K = \{z: |z|\leq \veps\}$. Combining this with Theorem~\ref{thm:AT-c-intro}  and \eqref{eq:AT-sharp} we obtain 
$$
	\exp\left( - A / \tc (K)\right) \leq T_\om(K) \leq  \exp\left( \frac{-1}{ [cap_\om(K)]^\frac{1}{n}} \right)
$$
with sharp exponents. This concludes the sharpness of the exponents in the second inequality in  Theorem~\ref{thm:Vigny}. So is the second one in Theorem~\ref{thm:AT-c-intro}.
\end{remark}

Next, we prove the existence of bounded solution to the Monge-Amp\`ere equation. Recall from \cite[Definition~1.3]{EGZ09} that the measure $\mu$ belongs to the class $\cH(\tau)$ with $\tau>0$ if
\[\notag
	\mu(K) \leq A [cap_\om(K)]^{1+\tau}
\]
for every Borel subset $K\subset X$.

\begin{proof}[Proof of Corollary~\ref{cor:ma-intro}] 
By \cite[Theorem~2.1]{EGZ09}, it enough to show that $\mu \in \cH(\tau_1)$ for some $\tau_1>0$. The proof of \cite[Lemma~3.3]{DKN} gives that for $u, v \in PSH(X,\om)$ and $v\leq u \leq 0$,
$$
	\|u - v\|_*  \leq  \|u-v\|_{L^2(X)} + A \|v\|_{L^\infty} \leq (A+1) \|v\|_{L^\infty},
$$
where $A = \int_X \om \wed \al^{n-1}$. Assume now $\sup_X v=0$ and $u= v_s:= \max\{v, -s +1\}$ for $s\geq 1$. By \cite[Eq. (3.5)]{DKN},
$$
	\| v_s -v \|_{L^1(X)} \leq  c_2 e^{-\tau_2 s}
$$
for uniform constants $c_2, \tau_2>0$. Using the $(W^*(X), \log^p)$-continuity of $\mu,$ we have
$$
	\mu(v\leq -s) \leq \int_X (v_s -v)d\mu  \leq c_3 \|v\|_{L^\infty} (\tau_2 s + \log\|v\|_{L^\infty} +1)
$$
for a uniform positive constant $c_3$. From this, we get  the statement of \cite[Lemma~3.7]{DKN} for $\om$ being a semi-positive and big form. Hence, the remained part follows the same lines  as the one of \cite[Theorem~3.6]{DKN}. 
\end{proof}

\end{document}